\documentclass{article}

\usepackage{amsmath,amssymb,amstext,amsfonts,latexsym,amsthm}
\usepackage{dsfont,mathrsfs,txfonts}
\usepackage{hyperref}
\usepackage{graphicx}
\usepackage{enumitem}
\usepackage[auth-sc,affil-sl]{authblk}
\newtheorem{theorem}{Theorem}
\newtheorem*{theorem*}{Favard Theorem}
\newtheorem{lemma}{Lemma}[section]
\newtheorem{proposition}{Proposition}[section]

\newtheorem{remark}{Remark}[section]
\newtheorem{example}{Example}[section]

\newcounter{Th-Alfa}
\newcommand{\CC}{\mathds{C}}
\newcommand{\ZZp}{\mathds{Z}_{+}}
\newcommand{\ZZi}{\overline{\mathds{Z}}_{+}}
\newcommand{\RRp}{\mathds{R}_{+}}

\newcommand{\PP}{\mathds{P}}

\newcommand{\MM}{\mathds{M}}
\newcommand{\HH}{\mathds{H}}
\newcommand{\RR}{\mathds{R}}
\newcommand{\ZZ}{\mathds{Z}}

\newcommand{\sopor}[1]{\mathrm{supp}\left(#1 \right)}

\newcommand{\dsty}{\displaystyle}
\newcommand{\tran}{\mathsf{T}}

\newcommand{\han}[1]{\mathcal{H}\left[#1 \right]}
\newcommand{\diag}[1]{\mathrm{Diag}\!\left(#1\right)}
\newcommand*{\doi}[1]{\href{http://dx.doi.org/#1}{DOI: #1}}
\title{Hessenberg-Sobolev matrices and Favard  type theorems}

\author[1]{H\'{e}ctor Pijeira-Cabrera\thanks{hpijeira@math.uc3m.es}}
\author[1]{Laura  Decalo-Salgado \thanks{100336220@alumnos.uc3m.es}} 
\author[2]{Ignacio P\'{e}rez-Yzquierdo\thanks{iperez58@uasd.edu.do}\thanks{Research partially supported by  Fondo Nacional de Innovaci\'{o}n  y Desarrollo Cient\'{\i}fico y Tecnol\'{o}gico (FONDOCYT),  Dominican Republic, under grant   {2016-2017-080} No. 013-2018.}}
\affil[1]{Departamento de Matem\'{a}ticas\\Universidad Carlos III de Madrid\\ Spain}
\affil[2]{Universidad Aut\'{o}noma de Santo Domingo\\ Dominican Republic}

\begin{document}

\maketitle

\begin{abstract}
 We study the relation  between certain non-degenerate lower Hessenberg infinite matrices $\mathcal{G}$ and the existence of sequences of orthogonal polynomials   with respect to Sobolev  inner products. In other words,  we extend the well-known Favard theorem for Sobolev orthogonality.   We characterize the structure of the matrix $\mathcal{G}$ and the   associated matrix of formal moments $\mathcal{M}_{\mathcal{G}}$ in terms of certain matrix operators.
\end{abstract}

\vspace{1cm} \noindent {\sl Keywords:} Sobolev orthogonality, orthogonal polynomials, moment problem, Favard theorem, Hessenberg matrices, Hankel matrices

\section{Introduction}

Let $\mathcal{A}=(a_{i,j})_{i,j=0}^{\infty}$   be  an infinite matrix of real numbers. For  $m\in \ZZ$,   we say that the entry $a_{i,j}$ lies in the $m$-diagonal if $j=i+m$. Obviously, the $0$-diagonal is the usual main diagonal of $\mathcal{A}$.     The matrix  $\mathcal{A}$ is an \emph{ $m$-diagonal matrix} if all of its nonzero elements lie in its $m$-diagonal and \emph{lower (upper) triangular  matrix} if  $a_{i,j}=0$ whenever $j>i$ ($j<i$). The symbols $\mathcal{A}^{\tran}$ and $[\mathcal{A}]_n$  denote the transposed matrix and the squared matrix of the first $n$ rows and columns of  $\mathcal{A}$, respectively.
$\mathcal{I}$  is called the \emph{unit matrix};   its $(i,j)$th entry is  $\delta_{i,j}$ where $\delta_{i,j}=0$ if $i\neq j$  and $\delta_{i,i}=1$. $\mathcal{A}$ is called \emph{positive definite of infinite order}  if $\det ([\mathcal{A}]_n)>0$  for all $n\geq 1$, where $\det ([\mathcal{A}]_n)$ is the determinant of  $[\mathcal{A}]_n$. If $\det ([\mathcal{A}]_n)>0$  for all $1 \leq n\leq k$ and $\det ([\mathcal{A}]_n)=0$  for all $ n> k$, we say that  $\mathcal{A}$ is a \emph{positive definite matrix of order $k$}.

According to   the definitions given in  \cite[Ch. II]{Coo50}, if  $\mathcal{A}$ and $\mathcal{B}$ are two   infinite matrices such that $\mathcal{A}\cdot \mathcal{B}=\mathcal{I}$, then $\mathcal{B}$  is called a right-hand  inverse of $\mathcal{A}$, denoted by $\mathcal{A}^{-1}$; and $\mathcal{A}$ is called a left-hand  inverse  of $\mathcal{B}$,  denoted by $^{-1}\mathcal{B}$. The transposed  of $\mathcal{A}^{-1}$  ($^{-1}\!\mathcal{A}$ ) and  $\mathcal{A}^{m}$\; (the $m$th power of the matrix $\mathcal{A}$, with $m \in \ZZp$)  are denoted by $\mathcal{A}^{-\tran}$ ($^{-\tran}\!\mathcal{A}$)  and $\mathcal{A}^{m\tran}$ respectively. Moreover,  $\mathcal{A}^{-m}= \left(\mathcal{A}^{-1}\right)^m$,  where  $m \in \ZZp$.

A difficulty of dealing with infinite matrices is that matrix products can be ill-defined (c.f. \cite[(1)]{CanMaMoVa}).  Nevertheless, in this paper we will only consider the product of   infinite matrices $\mathcal{A}=(a_{i,j})_{i,j=0}^{\infty}$ and $\mathcal{B}=(b_{i,j})_{i,j=0}^{\infty}$ when $\mathcal{A}\mathcal{B}=(\sum_{k} a_{i,k}b_{k,j})_{i,j=0}^{\infty}$ is  such that each  sum  ($i,j$-dependent)  involves only a finite number of non-null summands (c.f. \cite[Def. 1]{CanMaMoVa}).

We will denote by $\mathcal{U}$ the infinite matrix whose  $(i,j)$th entry is $\delta_{i+1,j}$ for  $i,j\in \ZZp$; i.e. the \emph{upper (or backward) shift infinite matrix} given by the expression
    \begin{equation*}\label{UpperShift}
 \mathcal{U}=\left( \begin{array}{cccc}
               0 & 1 & 0 & \cdots \\
0 & 0 & 1 &  \cdots \\
0 &0 & 0 &  \cdots \\
\vdots & \vdots & \vdots & \ddots
              \end{array}
     \right).
\end{equation*}
The matrix $\mathcal{U}^{\tran}$ is called   the \emph{lower (or forward) shift infinite matrix} and it is easy to check that $\dsty \mathcal{U}\cdot \mathcal{U}^{\tran}= \mathcal{I};$ i.e. $\mathcal{U}^{\tran}=\mathcal{U}^{-1}$, the right hand inverse of $\mathcal{U}$ (c.f. \cite[Sec. \;  0.9]{HorJoh13}).

An infinite  Hankel matrix  is an infinite matrix in which each ascending skew-diagonal from left to right is constant. In other words, $\mathcal{H}=(h_{i,j})_{i,j=0}^{\infty}$ is a Hankel matrix if $h_{i,j+1}=h_{i+1,j}$ for all $i,j \in \ZZp$ or equivalently if
\begin{equation}\label{HankelCond}
\mathcal{U} \mathcal{H}-\mathcal{H} \mathcal{U}^{-1}=\mathcal{O},
\end{equation}
where $\mathcal{O}$ denote  the infinite null matrix.  If  $\{ r_i\}_{i=0}^{\infty}$ is  a sequence of real numbers,  we denote $\diag{r_i}$ the infinite diagonal matrix whose $i$th main diagonal entry is $r_i$, and  by $\han{ r_i}$  the associated Hankel matrix,  defined as
  $$\han{ r_i }=\left( \begin{array}{cccc}
r_0 & r_1 &  r_2 & \cdots \\
r_1 & r_2 &  r_3 & \cdots \\
 r_2  & r_3  &  r_4  &  \cdots  \\
\vdots  & \vdots  &  \vdots &   \ddots
\end{array}\right).$$

We say that a matrix  $\mathcal{M}=(m_{i,j})_{i,j=0}^{\infty}$  is a \emph{Hankel-Sobolev matrix}  if there exists a sequence of Hankel matrices $\left\{\mathcal{H}_k\right\}_{k=0}^{\infty}$ such that
\begin{equation}\label{Hankel-Sobolev}
   \mathcal{M}= \sum_{k=0}^{\infty} \left(\mathcal{U}^{-k}\;\mathcal{D}_k \;  \mathcal{H}_{k} \;  \mathcal{D}_k\;\mathcal{U}^{k}\right),
\end{equation}
where $\mathcal{D}_0=\mathcal{I}  $ and  $\mathcal{D}_k=\diag{\frac{(k+i)!}{i!}}$  for each $k>0$, e.g. 
      $$   \mathcal{D}_1=\begin{pmatrix}1 & 0 & 0 & \cdots\\
0 & 2 & 0 & \cdots\\
0 & 0 & 3 &\cdots\\
\vdots &\vdots & \vdots & \ddots \end{pmatrix}, \; \mathcal{D}_2=\begin{pmatrix}2 & 0 & 0 &\cdots\\
0 & 6 & 0 & \cdots\\
0 & 0 & 12 & \cdots\\
\vdots &\vdots & \vdots &\ddots\end{pmatrix} \text{ and } \mathcal{D}_3=\begin{pmatrix}6 & 0 & 0 &\cdots\\
0 & 24 & 0 & \cdots\\
0 & 0 & 60 & \cdots\\
\vdots &\vdots & \vdots & \ddots\end{pmatrix}.      $$

We say that a Hankel-Sobolev matrix $\mathcal{M}$ is of \emph{index} $d \in \ZZp$ if $\mathcal{H}_d\neq\mathcal{O}$ and $\mathcal{H}_k=\mathcal{O}$ for all $k>d$. Otherwise, we will say that $\mathcal{M}$ is of \emph{infinite index}.

Let  $\mathcal{M}$ be a Hankel-Sobolev matrix of index  $d \in \ZZi=\ZZp \cup \{\infty\}$. If $\mathcal{H}_k\neq \mathcal{O}$ for all $k<d$, we say that $\mathcal{M}$   is \emph{non-lacunary} and   \emph{lacunary} in any other case.

Hankel-Sobolev matrices appeared for the first time in \cite{BaLoPi99,Pij98} in close connection with the moment problem for a Sobolev inner product. Some of the properties of this class of infinite matrices have also been studied in  \cite{Di08,MaSz00,MaSz01,PiQuiRo11,RoSa03,Zag05}.

Let   $\MM$  be the linear space of all infinite matrices of real numbers. For  each  $\eta \in \ZZp$ fixed, we denote by $\Phi( \cdot,\eta)$ the  operator   from $\MM$ to itself given by the expression
\begin{align}\label{Phi-operator}
\Phi( \mathcal{A},\eta):=&\sum_{\ell=0}^{\eta} (-1)^\ell \binom{\eta}{\ell}  \mathcal{U}^{\eta-\ell} \mathcal{A} \mathcal{U}^{-\ell},
\end{align}
where  $\mathcal{A} \in \MM$  and $\binom{\eta}{\ell} $ denote    binomial coefficients. Obviously,  $\Phi( \cdot,\eta)$ is a  linear operator.

One of the main results of this work is the following intrinsic  characterization of the Hankel-Sobolev matrices using the operator $\Phi( \cdot,\eta)$, which we will be prove in section \ref{SecMatrixOper}.

\begin{theorem} \label{ThCharact-HankelS}
An  infinite matrix  $\mathcal{M}$  is a Hankel-Sobolev matrix of index $d \in {\ZZp}$, if and only if  $\mathcal{M}$ is a symmetric matrix and
\begin{equation}\label{ThCharact-01}
\Phi( \mathcal{M},2d+1)=\mathcal{O} \quad \text{ and } \quad \Phi( \mathcal{M},2d)\neq\mathcal{O}.\end{equation}
Moreover, for $k=0,\,1,\dots,\;d$;  the Hankel matrix  $\mathcal{H}_{d-k}$ in \eqref{Hankel-Sobolev} is given by
\begin{equation}\label{ThCharact-02}
 \mathcal{H}_{d-k}= \frac{(-1)^{d-k}}{(2d-2k)!}\; \Phi( \mathcal{M}_{d-k},2 d-2k),
\end{equation}
where $\dsty \mathcal{M}_d=\mathcal{M}$ and  $\dsty \mathcal{M}_{d-k}=\mathcal{M}_{d-k+1}- \mathcal{U}^{-d-1+k}  \mathcal{D}_{d+1-k} \mathcal{H}_{d+1-k} \mathcal{D}_{d+1-k} \mathcal{U}^{d+1-k}$  for  $k=1,\;2,\dots, \;  d$.
\end{theorem}

An infinite matrix  $\dsty \mathcal{G}=(g_{i,j})_{i,j=0}^{\infty}$  is a \emph{lower Hessenberg infinite matrix}  if $g_{i,j}=0$ whenever $j-i>1$ and  at least one entry   of the $1$-diagonal is different from zero, i.e.
\begin{equation}\label{InfLowerHess}
\mathcal{G}= \left(\begin{array}{cccccc}
               g_{0,0} & g_{0, 1} & 0 &\cdots & 0 &\cdots\\
               g_{1,0} & g_{1,1} & g_{1, 2} &\cdots & 0 &\cdots\\
               \vdots &  \vdots &   \vdots & \ddots &  \vdots  &  \cdots\\
               g_{n,0} & g_{n,1} & g_{n,2} &\cdots & g_{n, n+1} & \cdots\\
               \vdots &  \vdots &  \vdots  & \vdots &  \vdots  &  \ddots
             \end{array}
 \right).
\end{equation}
Let us denote by $\HH$ the set of all lower Hessenberg infinite matrices. Additionally, if  $\mathcal{G} \in \HH$  and  all the entries in the $1$-diagonal are equal to $1$, we say that $\mathcal{G}$ is \emph{monic}. If all the entries of the   $1$-diagonal of  $\mathcal{G}$ are non-zero, we say that $\mathcal{G}$ is a   \emph{non-degenerate lower Hessenberg infinite matrix}  (for brevity hereafter referred as \emph{non-degenerate  Hessenberg  matrix}). An upper Hessenberg  matrix is a matrix whose transpose is  a lower Hessenberg matrix.

For  each  $\eta \in \ZZp$ fixed, we denote by $\Psi( \cdot,\eta)$ the  operator   from $\HH$ to $\MM$ given by the expression
\begin{align}  \label{Psi-operator}
\Psi(\mathcal{B}, \eta)=&\sum_{k=0}^{\eta}(-1)^{k}\, \binom{\eta}{k} \;\mathcal{B}^{k}\; \mathcal{B}^{(\eta-k)\; \tran},\; \mathcal{B} \in \HH.
\end{align}
Theorem \ref{RelaOperadores} establishes the relation between  the operators \eqref{Phi-operator} and  \eqref{Psi-operator}.

Given a  non-degenerate   matrix  $\mathcal{G}\in \HH$, we can generate a sequence of polynomials $\dsty \{Q_n\}_{n=0}^{\infty}$ as follows.  Assume that $Q_{0}(z) = t_{0,0}>0$, then
\begin{align}\label{SobolevRR01}
g_{0, 1} Q_{1}(x)  =& x Q_0(x)-g_{0,0} Q_0(x), \nonumber\\
g_{1, 2} Q_2(x) =& xQ_1(x)-g_{1,1} Q_1(x)-  g_{1,0} Q_0(x),  \nonumber\\
\vdots \qquad \vdots & \qquad \vdots  \nonumber\\
g_{n, n+1} Q_{n+1}(x)   =& x  Q_{n}(x) -    \sum_{k=0}^{n}g_{n,k} \, Q_{k}(x) ,\\
\vdots \qquad \vdots & \qquad \vdots  \nonumber
\end{align}
Hereafter, we will say that $\dsty \{Q_n\}$ is the \emph{sequence of polynomials generated by} $\dsty \mathcal{G}$. As  $\mathcal{G}$ is non-degenerate, $Q_n$ is a polynomial of degree $n$.

Let   $\mathcal{T}$  be the lower triangular  infinite matrix whose entries are the coefficients of the sequence of polynomials $\{Q_n\}$, i.e.
\begin{equation}\label{InfiniteMatForm02}
    \mathcal{Q}(x)= \mathcal{T}\;\mathcal{P}(x) \quad \text{ where } \quad
\mathcal{T}= \left(\begin{array}{ccccc}
               t_{0,0} & 0 & \cdots & 0  &\cdots\\
               t_{1,0} & t_{1, 1} & \cdots & 0  &\cdots\\
              \vdots &  \vdots &   \ddots &  \vdots  &\vdots\\
              t_{n,0} & t_{n,1} &  \cdots & t_{n, n} &\cdots   \\
               \vdots &  \vdots & \cdots &  \vdots  &\ddots
             \end{array} \right),
\end{equation}
$\mathcal{Q}(x)=(Q_0(x),Q_1(x),\cdots,Q_n(x), \cdots)^{\tran}$  and $\mathcal{P}(x)=(1,x,\cdots,x^n, \cdots)^{\tran}$.

As $\mathcal{G}$ is non-degenerate,   $t_{i,i}=t_{0,0}\,\left(\prod_{k=0}^{i-1} g_{k,k+1}\right)^{-1}\neq 0$ for all $i \geq 1$. Therefore,  there exists a  unique lower triangular  infinite matrix  $\mathcal{T}^{-1}$  such that $\mathcal{T}\cdot \mathcal{T}^{-1}=\mathcal{I}$ (c.f. \cite[(2.1.I)]{Coo50}), i.e. $\mathcal{T}$ has a unique right-hand inverse.  Furthermore, in this case $\mathcal{T}^{-1}$  is also a left-hand  inverse of  $\mathcal{T}$  and it is its only two-sided inverse (c.f. \cite[Remark (a) pp. 22]{Coo50}),
\begin{equation}\label{InfiniteMatForm03}
\mathcal{T}^{-1}= \left(\begin{array}{ccccc}
               \tau_{0,0} & 0 & \cdots & 0  &\cdots\\
               \tau_{1,0} & \tau_{1, 1} & \cdots & 0  &\cdots\\
              \vdots &  \vdots &   \ddots &  \vdots  &\vdots\\
              \tau_{n,0} & \tau_{n,1} &  \cdots & \tau_{n, n} &\cdots   \\
               \vdots &  \vdots & \cdots &  \vdots  &\ddots
             \end{array}
 \right), \; \text{where}\quad \tau_{i,i}=t_{i,i}^{-1}.\end{equation}

 We will denote by $\mathcal{M}_{\mathcal{G}}$ the \emph{matrix of formal moments associated with}   $\mathcal{G}$ (a non-de\-ge\-ne\-ra\-te Hessenberg matrix) defined by
\begin{equation}\label{GenMomentMatrix}
\mathcal{M}_{\mathcal{G}}= \left(m_{i,j}\right)_{i,j=0}^{\infty}=\mathcal{T}^{-1}\; \mathcal{T}^{-\tran}.
\end{equation}

We say that a non-degenerate  Hessenberg  matrix  $\mathcal{G}$ is a \emph{Hessenberg-Sobolev matrix of index $d \in {\ZZi}$} if its associated matrix of formal moments is a Hankel-Sobolev matrix of index $d$. In the following theorem, we give a characterization of these matrices.

\begin{theorem} \label{ThCharact-HessenS}
A  non-degenerate  Hessenberg  matrix  $\mathcal{G}$ is a Hessenberg-Sobolev matrix of index $d \in {\ZZp}$, if and only if   \begin{equation*}\label{ThCharact-HessenS-01}
\Psi( \mathcal{G},2d+1)=\mathcal{O} \quad \text{ and } \quad \Psi( \mathcal{G},2d)\neq\mathcal{O}.\end{equation*}
\end{theorem}
 The proof of this theorem is an immediate consequence of Theorem  \ref{ThCharact-HankelS} and  Theorem \ref{RelaOperadores} (stated in section \ref{Sect-Hessenberg}).

Let $\PP$ be   the linear space of polynomials with real coefficients, $\mathcal{G}$ be a non-de\-ge\-ne\-ra\-te  Hessenberg  matrix  and $\mathcal{M}_{\mathcal{G}}$  its associated matrix of formal moments.  If  $p(x)=\sum_{i=0} ^{n_1} a_i x^i$ and  $q(x)=\sum_{j=0} ^{n_2} b_j x^j$ are two  polynomials in  $\PP$ of degree $n_1$  and $n_2$ respectively. Then the bilinear form
           \begin{eqnarray}\label{MatInnerProd}
\langle p,q \rangle_{\mathcal{G}}=  (a_0,\cdots,a_{n_1},0,\cdots) \mathcal{M}_{\mathcal{G}} (b_0,\cdots,b_{n_2},0,\cdots)^{\tran} = \sum_{i=0} ^{n_1} \sum _{j=0} ^{n_2} a_i \,m_{i,j}\,{b_j};
     \end{eqnarray} defines an inner product on $\PP$ and $\|\cdot \|_{\mathcal{G}}=\sqrt{\langle \cdot ,\cdot  \rangle_{\mathcal{G}}}$  is  the norm induced by \eqref{MatInnerProd} on $\PP$ (c.f. Theorem \ref{Th-FormalOP}).

Let $d\in\ZZp$ and $\vec{\mathbf{\mu}}_d= (\mu_0,\mu_1,\dots,\mu_d)$ be a vector of $d+1$ measures, we write $\vec{\mathbf{\mu}}_d \in \mathfrak{M}_{d}(\RR)$ if for each $k$ ($0\leq k \leq d$) the measure $\mu_k$ is a non-negative finite Borel measure with support $\Delta_k \subset \RR$, $\PP\subset \mathbf{L}_1\left(\mu_k\right)$, $\mu_0$ is positive and $\Delta_0$ contains infinite points. If  $d=\infty$ and  $\vec{\mathbf{\mu}}_{\infty}$  is a sequence of  measures that satisfy the above conditions, we write $\vec{\mathbf{\mu}}_{\infty}\in \mathfrak{M}_{\infty}(\RR)$.

For $d\in\ZZi$ and   $\vec{\mathbf{\mu}}_d \in \mathfrak{M}_{d}(\RR)$ , we define  on $\PP$ the Sobolev inner product
\begin{equation}
\label{Sob-InnerP}
\langle f, g \rangle_{\vec{\mathbf{\mu}}_d} = \sum_{k=0}^{d} \int f^{(k)}(x) g^{(k)}(x) d\mu_{k}(x) = \sum_{k=0}^{d} \langle f^{(k)}, g^{(k)}\rangle_{\mu_k}\;,
\end{equation}
where $f,g \in \PP$ and $f^{(k)}$ denote  the $k$th derivative of  $f$. The symbol $\|\cdot\|_{\vec{\mathbf{\mu}}_d}=\sqrt{\langle \cdot, \cdot \rangle_{\vec{\mathbf{\mu}}_d}}$ denotes the Sobolev norm associated with \eqref{Sob-InnerP}.
Note that although   $d$ is usually considered a non-negative  integer (c.f. \cite{MaXu15}), the case $d=\infty$  has sense  on  $\PP$. If all the measures $\mu_k$ involved in \eqref{Sob-InnerP} are positive, we say that the Sobolev inner product is \emph{non-lacunary} and   \emph{lacunary} in any other case.

Taking into account the nature of the support of the measures involved in \eqref{Sob-InnerP}, we have the following three cases:
\begin{description}
  \item[Continuous case.]  The measures $\mu_0, \cdots, \mu_d$ are supported on   infinite subsets.
  \item[Discrete case.]  The support of the measure $\mu_0$ is an infinite subset and the measures $\mu_1, \cdots, \mu_d$ are supported on finite subsets.
\item[Discrete-continuous case.]  The support of the measure $\mu_d$ is an infinite subset and the measures $\mu_0, \cdots, \mu_{d-1}$ are supported on finite subsets.
\end{description}

The notion of Sobolev moment and several related  topics  were  firstly introduced in  \cite{BaLoPi99}.   The \emph{$(n,k)$-moment associated with the inner product } \eqref{Sob-InnerP}   is defined as $s_{n,k}=\langle x^n,x^k\rangle_{\vec{\mathbf{\mu}}_d}$ ($n,k\geq 0$), provided the integral exists.   In \cite{BaLoPi99}, it was proved  that  the  infinite matrix of moments $\mathcal{S}$ with entries $s_{n,k}$, ($n,k\geq 0$) is a Hankel-Sobolev matrix  (c.f. \cite{BaLoPi99} and Subsection \ref{Sec-S-MomentP} of this paper). Furthermore, if ${Q_n}$ is the sequence of orthonormal polynomials with respect to \eqref{Sob-InnerP} with leading coefficient $c_{n} >0$, then the infinite matrix $\mathcal{G}_{\vec{\mathbf{\mu}}_d}$ with entries $g_{i,j}=\langle xQ_i, Q_j \rangle_{\vec{\mathbf{\mu}}_d} $  is a   non-degenerate  Hessenberg  matrix. In this case,  the sequence of orthonormal polynomials ${Q_n}$ is the sequence of polynomials generated by $\mathcal{G}_{\vec{\mathbf{\mu}}_d}$.

The following theorem gives a characterization of  the non-degenerate  Hessenberg  matrices whose  sequence of generated polynomials is ortogonal with respect to a So\-bo\-lev inner product as \eqref{Sob-InnerP}.

\begin{theorem} [Favard  type theorem for continuous case] \label{ThFavardSobolev} Let $\mathcal{G}$ be a non-degenerate  Hessenberg  matrix. Then, there exists  $d \in \ZZp$ and $\vec{\mathbf{\mu}}_d \in \mathfrak{M}_{d}(\RR)$  such that $ \langle p, q \rangle_{\vec{\mathbf{\mu}}_d} =\langle p, q \rangle_{\mathcal{G}} $ if and only if \;
\begin{enumerate}
  \item $\mathcal{G}$ is a Hessenberg-Sobolev matrix of index $d \in  \ZZp$.
  \item For each  $k=0,\,1,\dots,\;d$; the Hankel matrix    $\mathcal{H}_{d-k}$  defined by  \eqref{ThCharact-02}, is a positive definite matrix of infinite order.
\end{enumerate}
\end{theorem}

The  Favard  type theorems for the cases discrete and the discrete-continuous   are   Theorems \ref{ThFavardSobolevDiscrete} and \ref{ThFavardSobolevDiscreteCont}, respectively. Some basic aspects about the classical moment problem and the Sobolev moment problem are revisited  in subsection \ref{Sec-S-MomentP}.

In Section \ref{SecMatrixOper}, we proceed with the study of  the properties of the matrix operator $\Phi( \cdot,\eta)$, the Hankel-Sobolev matrices and the proof of Theorem \ref{ThCharact-HankelS}. We revisit the Sobolev moment problem  in  subsection \ref{Sec-S-MomentP}. The third section  is devoted to  study  the properties of the bilinear form \eqref{MatInnerProd}  and the nexus between the operators  $\Phi( \cdot,\eta)$ and $\Psi( \cdot,\eta)$. In the last section, we prove the extension of the Favard Theorem for Sobolev orthogonality stated in Theorem \ref{ThFavardSobolev}.

\section{Hankel-Sobolev matrices}\label{SecMatrixOper}

First of all, we need to prove that the notion of a Hankel-Sobolev matrix introduced in \eqref{Hankel-Sobolev}  is well-defined.

\begin{proposition} \label{Uniqueness-HankelS} Let  $\mathcal{M}$ be a Hankel-Sobolev matrix, then the decomposition of $\mathcal{M}$  established in \eqref{Hankel-Sobolev} is unique.
\end{proposition}

\begin{proof} We first recall that for each $k \in \ZZp$,  $\mathcal{D}_k$ is a diagonal matrix with positive entries in the main diagonal. Furthermore, if $\mathcal{A}$ is an infinite matrix and $k\in \ZZp$ is fixed, the matrix $\left(\mathcal{U}^{-k}\; \mathcal{A} \;\mathcal{U}^{k}\right)$ is obtained adding to $\mathcal{A}$  the first $k$ rows and columns of zeros.

Suppose there are two sequences of Hankel matrices,  $\left\{\mathcal{H}_k\right\}_{k=0}^{\infty}$ and  $\left\{\widehat{\mathcal{H}}_k\right\}_{k=0}^{\infty}$,  such that
$$
 \mathcal{M}=   \sum_{k=0}^{\infty} \left(\mathcal{U}^{-k}\;\mathcal{D}_k \;  \mathcal{H}_{k} \;  \mathcal{D}_k\;\mathcal{U}^{k}\right) \quad \text{ and }\quad
 \mathcal{M}=     \sum_{k=0}^{\infty} \left(\mathcal{U}^{-k}\;\mathcal{D}_k \;  \widehat{\mathcal{H}}_{k} \;  \mathcal{D}_k\;\mathcal{U}^{k}\right).
$$
Therefore,
$$\sum_{k=0}^{\infty} \left(\mathcal{U}^{-k}\;\mathcal{D}_k \;   \left(\mathcal{H}_{k} -\widehat{\mathcal{H}}_{k} \right)\;  \mathcal{D}_k\;\mathcal{U}^{k}\right)=\mathcal{O}.$$
Hence,  for each  $k\in \ZZp$ fixed, the matrix $\left(\mathcal{H}_{k} -\widehat{\mathcal{H}}_{k} \right)$ is a Hankel matrix whose first row has all its entries equal to zero, i.e. $ \mathcal{H}_{k} =\widehat{\mathcal{H}}_{k}$, which  completes the proof.
\end{proof}

Obviously, the matrix operator $\Phi( \cdot,\eta)$ defined in \eqref{Phi-operator} is linear. Before  proving   Theorem \ref{ThCharact-HankelS}, we need to study some other properties of this operator and some  auxiliary results.

\begin{proposition}[Recurrence] \label{RecF-1}   Let $\eta \in \ZZp$ fixed and  $\mathcal{A} \in \MM$ ,  then
\begin{equation}\label{RecuRelat}
 \Phi( \mathcal{A},\eta+1)=\mathcal{U}\,\Phi( \mathcal{A},\eta)-\Phi( \mathcal{A},\eta) \,\mathcal{U}^{-1} .
\end{equation}
\end{proposition}

\begin{proof}
\begin{align*}
 \Phi( \mathcal{A},\eta+1)  = &   \mathcal{U}^{\eta+1} \mathcal{A} + \left( \sum_{\ell=1}^{\eta} (-1)^\ell  \binom{\eta+1}{\ell} \mathcal{U}^{\eta+1-\ell} \mathcal{A} \mathcal{U}^{-\ell} \right) +   (-1)^{\eta+1}  \mathcal{A} \mathcal{U}^{-\eta-1}\\
  = &  \mathcal{U}^{\eta+1} \mathcal{A} + \left( \sum_{\ell=1}^{\eta} (-1)^\ell \left( \binom{\eta}{\ell-1}+ \binom{\eta}{\ell}\right)  \mathcal{U}^{\eta+1-\ell}  \mathcal{A} \mathcal{U}^{-\ell} \right)    +   (-1)^{\eta+1}  \mathcal{A} \mathcal{U}^{-\eta-1}\\
   =&  \sum_{\ell=0}^{\eta} (-1)^\ell \binom{\eta}{\ell}  \mathcal{U}^{\eta+1-\ell}\mathcal{A} \mathcal{U}^{-\ell} +\sum_{\ell=1}^{\eta+1} (-1)^{\ell} \binom{\eta}{\ell-1}  \mathcal{U}^{\eta+1-\ell} \mathcal{A} \mathcal{U}^{-\ell}\\  =&  \mathcal{U} \left( \sum_{\ell=0}^{\eta} (-1)^\ell \binom{\eta}{\ell}  \mathcal{U}^{\eta-\ell} \mathcal{A} \mathcal{U}^{-\ell} \right) -\left(\sum_{\ell=0}^{\eta} (-1)^\ell \binom{\eta}{\ell}  \mathcal{U}^{\eta-\ell} \mathcal{A} \mathcal{U}^{-\ell}\right) \mathcal{U}^{-1}\\
   =& \mathcal{U} \, \Phi( \mathcal{A},\eta)-\Phi( \mathcal{A},\eta) \mathcal{U}^{-1}
\end{align*}
\end{proof}

The following proposition is an immediate consequence of   Proposition   \ref{RecF-1} and \eqref{HankelCond}.

\begin{proposition} \label{PropHankel}
 If for a matrix $\mathcal{A} \in \MM$,  there exists $\eta \in \ZZp$ such that $\Phi( \mathcal{A},\eta)=\mathcal{O}$, then
   \begin{enumerate}
     \item[(a)]  $\Phi( \mathcal{A},\eta_1)=\mathcal{O}$ for all $\eta_1 \geq \eta.$
     \item[(b)]   For all $c\in \RR$ and $\eta \geq 1$ the matrix $c\,\Phi( \mathcal{A},\eta-1)$ is a Hankel matrix.
   \end{enumerate}
\end{proposition}

\begin{proposition} \label{lemma-symmetric} Assume that $\mathcal{A}\in \MM$ is a symmetric matrix, then  $\Phi( \mathcal{A},\eta)$ is a symmetric (antisymmetric) matrix if and only  if $\eta$ is an even (odd)  integer number.
\end{proposition}

\begin{proof}
\begin{align*}
\left(\Phi( \mathcal{A},\eta)\right)^{\tran}=&\sum_{\ell=0}^{\eta} (-1)^\ell \binom{\eta}{\ell}  \left(\mathcal{U}^{\eta-\ell} \mathcal{A} \mathcal{U}^{-\ell} \right)^{\tran}= \sum_{\ell=0}^{\eta} (-1)^\ell \binom{\eta}{\eta-\ell}  \mathcal{U}^{\ell}  \mathcal{A} \mathcal{U}^{\ell-\eta} \\
= & \sum_{\ell=0}^{\eta} (-1)^{\eta-\ell} \binom{\eta}{\ell}  \mathcal{U}^{\eta-\ell}  \mathcal{A} \mathcal{U}^{-\ell}=  (-1)^{\eta}  \sum_{\ell=0}^{\eta} (-1)^{\ell} \binom{\eta}{\ell}  \mathcal{U}^{\eta-\ell}  \mathcal{A} \mathcal{U}^{-\ell} \\ = & (-1)^{\eta}  \Phi( \mathcal{A},\eta).\\
\end{align*}
\end{proof}

\begin{theorem} \label{Th-Suff}  Let $d\in \ZZi$  and $\mathcal{M}$ be a Hankel-Sobolev matrix of index $d$, as in  \eqref{Hankel-Sobolev}. Denote
$$
   \mathcal{M}_\eta= \sum_{k=0}^\eta \mathcal{U}^{-k}  \mathcal{D}_k   \mathcal{H}_{k}    \mathcal{D}_k  \mathcal{U}^{k}, \quad 0\leq \eta \leq d.
$$
 Then
\begin{enumerate}
  \item[(a)]$\Phi( \mathcal{M}_\eta,2\eta+1)=\mathcal{O}$.
  \item[(b)] $\dsty \mathcal{H}_\eta= \frac{(-1)^\eta}{(2\eta )!}\Phi( \mathcal{M}_\eta,2\eta)$.
\end{enumerate}
\end{theorem}

Before proving the previous theorem, we need the next two lemmas. The first one,  is a version of the famous Euler’s Finite Difference Theorem (c.f. \cite[Sec. 6.1]{QuGo16}).

\begin{lemma}\label{Euler_Lem} Let $\dsty f(z)$ be a complex polynomial of degree $n$ and leading coefficient $a_n \in \CC$. Then, for all $m,\nu\in \ZZp$

$$\sum_{\ell=0}^{\nu} (-1)^{\ell} \binom{\nu}{\ell} f(\ell)=\left\{\begin{array}{rl}
0, &\text{if } 0 \leq n < \nu,  \\
(-1)^{\nu } \;\nu! \; a_ \nu, &\text{if }  n=\nu.
\end{array}\right.$$
\end{lemma}

\begin{lemma}\label{SymmMatrix} Let $\mathcal{A}=[a_{i,j}]$ be an  infinite symmetric matrix, whose $(i,j)$ entry is $a_{i,j}$.  Then, for all $\eta,\nu \in \ZZp$, the matrix $\left(\mathcal{U}^{\eta} \,\mathcal{A}\,\mathcal{U}^{-\nu}\right)$ is symmetric.
\end{lemma}
\begin{proof} \

Given a  sequence  of double indexes $\{a_{i,j}\}$ and two functions $f,\,g$ on $\ZZp$,  we denote by $\mathcal{A}=[a_{f(i),g(j)}]$ the corresponding infinite matrix, whose $(i,j)$ entry is $a_{f(i),g(j)}$. Therefore, as $a_{i,j}=a_{j,i}$ we get
$$
 \mathcal{U}^{\eta} \, \mathcal{A}\,\mathcal{U}^{-\nu}=  \left[a_{\eta+i,\nu+j}\right]=\left[a_{\nu+j,\eta+i}\right]=\mathcal{U}^{\nu} \, \mathcal{A}\,\mathcal{U}^{-\eta}=\left( \mathcal{U}^{\eta} \, \mathcal{A}\,\mathcal{U}^{-\nu}\right)^{\tran}, \; i,j=1,2,\dots\,.
$$
\end{proof}

\begin{proof}[Proof of Theorem \ref{Th-Suff}]  \

Let $0\leq k \leq d$ fixed and $\dsty \mathcal{R}_k=\mathcal{U}^{-k}\mathcal{D}_k    \mathcal{H}_{k}   \mathcal{D}_k\mathcal{U}^{k} $. Then $\dsty \mathcal{M}_\eta=  \sum_{k=0}^\eta \mathcal{R}_k $ and from linearity $\dsty  \Phi( \mathcal{M}_\eta,2\eta+1)= \sum_{k=0}^\eta  \Phi \left(\mathcal{R}_k,2\eta+1 \right)$.
\begin{align}
\nonumber  \Phi(\mathcal{R}_k,2\eta+1)=& \sum_{\ell=0}^{2\eta+1} (-1)^\ell \binom{2\eta+1}{\ell}  \mathcal{U}^{2\eta+1-\ell}\mathcal{R}_k\mathcal{U}^{-\ell} \\
\nonumber    =& \sum_{\ell=0}^{\eta} (-1)^\ell \binom{2\eta+1}{\ell}  \mathcal{U}^{2\eta+1-\ell}\mathcal{R}_k\mathcal{U}^{-\ell} +  \sum_{\ell=\eta+1}^{2\eta+1} (-1)^\ell \binom{2\eta+1}{\ell}  \mathcal{U}^{2\eta+1-\ell}\mathcal{R}_k\mathcal{U}^{-\ell} \\
\nonumber      =& \sum_{\ell=0}^{\eta} (-1)^\ell \binom{2\eta+1}{\ell}  \mathcal{U}^{2\eta+1-\ell}\mathcal{R}_k\mathcal{U}^{-\ell} \\
\nonumber  &+   \sum_{\ell=\eta+1}^{2\eta+1} (-1)^\ell \binom{2\eta+1}{2\eta+1-\ell}  \mathcal{U}^{2\eta+1-\ell}\mathcal{R}_k\mathcal{U}^{-\ell} \\
\nonumber        =& \sum_{\ell=0}^{\eta} (-1)^\ell \binom{2\eta+1}{\ell}  \; \mathcal{U}^{2\eta+1-\ell}\mathcal{R}_k\mathcal{U}^{-\ell} -  \sum_{\ell=0}^{\eta} (-1)^\ell \binom{2\eta+1}{\ell}  \; \mathcal{U}^{\ell}\mathcal{R}_k\mathcal{U}^{\ell-2\eta-1} \\
\label{Th-Suff-01}        =& \sum_{\ell=0}^{\eta} (-1)^\ell \binom{2\eta+1}{\ell}  \; \left(\mathcal{U}^{2\eta+1-\ell}\mathcal{R}_k\mathcal{U}^{-\ell} - \mathcal{U}^{\ell} \mathcal{R}_k\mathcal{U}^{\ell-2\eta-1} \right).
 \end{align}

Clearly, $\dsty \mathcal{R}_k$ is a symmetric matrix. Therefore, from Lemma \ref{SymmMatrix},   $ \dsty \mathcal{U}^{2\eta+1-\ell} \mathcal{R}_k\mathcal{U}^{-\ell} $ is a symmetric matrix too and
\begin{equation}\label{Th-Suff-02}
 \mathcal{U}^{2\eta+1-\ell} \mathcal{R}_k\mathcal{U}^{-\ell}  =   \left(  \mathcal{U}^{2\eta+1-\ell} \mathcal{R}_k\mathcal{U}^{-\ell} \right)^{\tran}= \mathcal{U}^{\ell} \mathcal{R}_k\mathcal{U}^{\ell -2\eta-1}.
\end{equation}
Combining  \eqref{Th-Suff-01}-\eqref{Th-Suff-02}    we get (a) in Theorem \ref{Th-Suff}.

From (a) and Proposition \ref{RecF-1}, it follows that
\begin{align*}
  \Phi( \mathcal{M}_\eta,2\eta)= & \Phi( \mathcal{R}_\eta,2\eta) +  \sum_{k=0}^{\eta-1}   \Phi( \mathcal{R}_k,2\eta) \\
    = &  \Phi( \mathcal{R}_\eta,2\eta) +  \mathcal{U}  \sum_{k=0}^{\eta-1}  \Phi(  \mathcal{R}_k,2\eta-1)- \sum_{k=0}^{\eta-1} \Phi( \mathcal{R}_k,2\eta-1) \mathcal{U}^{-1}\\
      = &  \Phi( \mathcal{R}_\eta,2\eta) +  \mathcal{U} \, \Phi( \mathcal{M}_{\eta-1},2\eta-1)- \Phi( \mathcal{M}_{\eta-1},2\eta-1) \mathcal{U}^{-1}= \Phi( \mathcal{R}_\eta,2\eta) .\\
    \Phi( \mathcal{R}_\eta,2\eta)  = & \sum_{\ell=0}^{2\eta} (-1)^\ell  \binom{2\eta}{\ell}  \mathcal{U}^{2\eta-\ell}     \mathcal{R}_\eta  \mathcal{U}^{-\ell}= \sum_{\ell=0}^{2\eta} (-1)^\ell  \binom{2\eta}{\ell}  \mathcal{U}^{\eta-\ell}  \mathcal{D}_\eta    \mathcal{H}_{\eta}   \mathcal{D}_\eta\mathcal{U}^{\eta-\ell}.
  \end{align*}

   Let  $\{m_{k,i}\}$ be the sequence of real numbers such that $m_{k,i+j-2}$ is the $(i,j)$th entry of the infinite  Hankel matrix $ \mathcal{H}_{k}$, where $i,j=1,2,\dots$ In that case,  we write  $ \dsty  \mathcal{H}_{k}=\left[m_{k,i+j-2}\right]$. Thus,
  \begin{align*}
\mathcal{D}_\eta    \mathcal{H}_{\eta}   \mathcal{D}_\eta =&\left[\frac{(\eta+i-1)!}{(i-1)!}\frac{(\eta+j-1)!}{(j-1)!}m_{\eta,i+j-2}\right]  = \left(\eta! \right)^2\, \left[ \binom{\eta+i-1}{\eta}\binom{\eta+j-1}{\eta}m_{\eta,i+j-2}\right],
  \end{align*}
 and therefore
 \begin{align}\nonumber
    \Phi( \mathcal{R}_\eta,2\eta)  = &  \left(\eta! \right)^2   \left[\left(\sum_{\ell=0}^{2\eta}  (-1)^{\ell}\binom{2\eta}{\ell}\binom{\eta+i-1+(\eta-\ell)}{\eta}\binom{\eta+j-1-(\eta-\ell)}{\eta} \right)\;m_{\eta,i+j-2}\right] \\ \label{Th-Suff-1}
   = &  \left(\eta! \right)^2\,\left[\left(\sum_{\ell=0}^{2\eta}  (-1)^{\ell}\binom{2\eta}{\ell}\binom{2\eta+i-1-\ell}{\eta}\binom{j-1+\ell}{\eta} \right)\;m_{\eta,i+j-2}\right] .
 \end{align}

 Clearly  $\dsty f(\ell)= \binom{2\eta+i-1-\ell}{\eta}\binom{j-1+\ell}{\eta}$ is a polynomial of degree $2\eta$  in  $\ell$  and leading coefficient $\dsty \frac{(-1)^\eta}{(\eta!)^2} $. By Lemma \ref{Euler_Lem}, we deduce that

 \begin{equation}\label{Th-Suff-2}
\sum_{\ell=0}^{2\eta}  (-1)^{\ell}\binom{2\eta}{\ell}\binom{2\eta+i-1-\ell}{\eta}\binom{j-1+\ell}{\eta}=(-1)^{\eta}\; \binom{2\eta}{\eta}.
 \end{equation}
Hence, from \eqref{Th-Suff-1}-\eqref{Th-Suff-2} we get $\dsty  \Phi( \mathcal{R}_\eta,2\eta)  =(-1)^\eta (2\eta)![m_{\eta,i+j-2}]=(-1)^\eta\;(2\eta )!\,\mathcal{H}_\eta$ and (b). \end{proof}

We will assume that $\mathcal{A}$ is an infinite symmetric matrix because  this is obviously a necessary condition for \eqref{Hankel-Sobolev} to take place since the Hankel matrices $\mathcal{H}_k$ are symmetric.

\begin{theorem} 
Let $\mathcal{A}$ be an infinite symmetric matrix, $\eta \in \ZZp$ (fixed) such that  $\Phi( \mathcal{A},2\eta+1)=\mathcal{O}$. Then
\begin{equation}\label{Th-Suff3}
\Phi( \mathcal{A}_\eta,2\eta-1)=\mathcal{O},
\end{equation}
 where $ \mathcal{A}_\eta=\mathcal{A}-\mathcal{R}_\eta$ and  $\dsty \mathcal{R}_\eta= \frac{(-1)^\eta}{(2\eta)!}\mathcal{U}^{-\eta}  \mathcal{D}_\eta\Phi( \mathcal{A},2\eta) \mathcal{D}_\eta \mathcal{U}^{\eta}$.
\end{theorem}
\begin{proof} If $\mathcal{A}_\eta=\mathcal{O}$,  the theorem is obvious. Assume that $\mathcal{A}_\eta\neq \mathcal{O}$,  from Theorem \ref{Th-Suff}, we get  $\Phi( \mathcal{R}_\eta,2\eta)= \frac{(-1)^\eta}{(2\eta)!}\Phi( \mathcal{A},2\eta)$, i.e. $\Phi(  \mathcal{A}_\eta,2\eta)= \mathcal{O}$. According to  the recurrence formula \eqref{RecuRelat}, we have $ \mathcal{U}\,\Phi( \mathcal{A}_\eta,2\eta-1)  = \Phi( \mathcal{A}_\eta,2\eta-1)\mathcal{U}^{-1}$ which is equivalent to stating that $\Phi( \mathcal{A}_\eta,2\eta-1)$  is a Hankel matrix and therefore it is a symmetric matrix.

 On the other hand, $\mathcal{A}_\eta$  is a symmetric matrix since it is the difference of two symmetric matrices. Hence, from Proposition \ref{lemma-symmetric} we get that $\Phi( \mathcal{A}_\eta,2\eta-1)$ is antisymmetric which establishes  \eqref{Th-Suff3}.
\end{proof}

\begin{proof}[Proof of Theorem \ref{ThCharact-HankelS}] From  Theorem \ref{Th-Suff}, a Hankel-Sobolev matrix of index $d\in \ZZp$ satisfies the conditions \eqref{ThCharact-01} and each  Hankel matrix $\mathcal{H}_k$ holds   \eqref{ThCharact-02}, which establishes the  first implication of the theorem.

For the converse, assume that  $\mathcal{M}$ is a symmetric infinite matrix and there exits  $d\in \ZZp$ such that the conditions \eqref{ThCharact-01} are satisfied. From Proposition \ref{PropHankel}, $H_d=\frac{(-1)^d}{(2d)!}\Phi( \mathcal{M},2d)\neq \mathcal{O}$ is a Hankel matrix.

Denote  $\mathcal{M}_{d-1}=\mathcal{M}_{d}-\mathcal{R}_d$, where $\mathcal{M}_{d}=\mathcal{M}$ and  $\dsty \mathcal{R}_d= \mathcal{U}^{-d}  \mathcal{D}_d \mathcal{H}_d \mathcal{D}_d \mathcal{U}^{d}$. From \eqref{Th-Suff3}  and Proposition \ref{PropHankel}, $H_{d-1}=\frac{(-1)^{d-1}}{(2d-2)!}\Phi( \mathcal{M}_{d-1},2d-2)$ is a Hankel matrix.

 Let $\mathcal{M}_{d-k}=\mathcal{M}_{d+1-k}-\mathcal{R}_{d+1-k}$ and $\dsty \mathcal{R}_{d+1-k}= \mathcal{U}^{-d-1+k}  \mathcal{D}_{d+1-k} \mathcal{H}_{d+1-k} \mathcal{D}_{d+1-k} \mathcal{U}^{d+1-k}$. Repeating the previous argument, we get  that   $H_{d-k}=\frac{(-1)^{d-k}}{(2d-2k)!}\Phi( \mathcal{M}_{d-k},2d-2k)$ is a Hankel matrix   for $k=2,\dots, d$.

By construction,  it is clear that
$$\mathcal{M}=\sum_{k=0}^{d}\mathcal{R}_{d-k}= \sum_{k=0}^{d} \mathcal{U}^{k-d}  \mathcal{D}_{d-k} \mathcal{H}_{d-k} \mathcal{D}_{d-k} \mathcal{U}^{d-k},$$
i.e. $\mathcal{M}$ is a Hankel-Sobolev matrix and the proof is complete.
\end{proof}
\subsection{The Sobolev moment problem} \label{Sec-S-MomentP}

 Let $\mu$ be a finite positive Borel measure supported on the real line and $\mathbf{L}_2(\mu)$ be the usual Hilbert space of square integrable functions with respect to  $\mu$ with the inner product
\begin{equation}\label{L2-InnerP}
  \langle f,g \rangle_{\mu}= \int_{\RR} f(x) g(x) d\mu(x), \quad \mbox{for all } f,g \in \mathbf{L}_2(\mu).\end{equation}
The $n$th moment associated with the inner product \eqref{L2-InnerP} (or the measure $\mu$) is defined as $m_n=\langle x^n,1\rangle_{\mu}$ ($n\geq 0$), provided the integral exists. The Hankel matrix $\han{ m_n }$ is called the  \emph{matrix of moments} associated with $\mu$.

The classical moment problem  consists in solving the following question: given an arbitrary sequence of real numbers $\{m_n\}_{n\geq 0}$ (or equivalently the associated  Hankel matrix $\han{ m_n }$)  \,  and a closed subset $\Delta \subset \RR$, find a positive  Borel  measure $\mu$ supported on   $\Delta$,  whose $n$th moment is $m_n$, i.e.
\begin{equation*}\label{MomEstandasDef-1}
  m_n= \int_{\Delta} x^n d\mu(x), \quad \text{ for all $n\geq 0$}.
\end{equation*}

 It is said that the moment problem  $(\mathcal{H};\Delta)$ is \emph{definite},  if it has at least one solution and  \emph{determinate}  if the solution is unique. There are three named \emph{classical moment problems}: the \emph{Hamburger moment problem} when  the support of $\mu$ is on the whole real line,  the \emph{Stieltjes moment problem} if $\Delta=[0,\infty)$,  and the \emph{Hausdorff moment problem} for a bounded interval $\Delta$ (without loss of generality,  $\Delta=[0, 1]$).

As H. J. Landau write in the introduction of \cite[p.1]{Lan87}:  \emph{``The moment problem is a classical question in analysis, remarkable not only for its own elegance, but also for the extraordinary range of subjects, theoretical and applied, which it has illuminated''}. For  more details on the classical moment problem,  the reader is referred to \cite{Akh65,GoMe10,Lan87,ShoTam63,Sch17} and for historical aspects to \cite{Kje93} or \cite[Sec. 2.4]{GoMe10}.

Without  restriction  of generality,  we now turn our attention to the Hamburger moment problem referring to the following lemma as a necessary and sufficient condition for the problem of moments to be defined and determined.

\begin{lemma}[{\cite[Th. 1.2]{ShoTam63}}]\label{SchoTam-lemma} Let $\{ m_n\}_{n=0}^{\infty}$ be a sequence of real numbers and denote by $\mathcal{H}=\han{ m_n }$  the associated Hankel matrix.  Then
\begin{enumerate}
  \item  The Hamburger moment problem  $(\mathcal{H};\RR)$ has a solution,  whose support  is not reducible to a finite set of points, if and only if $\mathcal{H}$ is a matrix positive definite of infinite order (i.e. $\det ([\mathcal{H}]_n)>0$  for all $n\geq 1$).
  \item   The Hamburger moment problem  $(\mathcal{H};\RR)$ has a solution, whose support  consists of precisely $k$ distinct points, if and only if  $\mathcal{H}$ is a matrix positive definite of  order $k$ (i.e. $\det ([\mathcal{H}]_n)>0$  for all $1 \leq n\leq k$, and  $\det ([\mathcal{H}]_n)=0$  for all $ n> k$). The moment problem is determined in this case.
\end{enumerate}
\end{lemma}

The analogous results for the moment problem  of Stieltjes  $(\mathcal{H};\RRp)$ or the moment problem  of Hausdorff   $(\mathcal{H};[0,1])$ are \cite[Th. 1.3]{ShoTam63} and \cite[Th. 1.5]{ShoTam63} respectively. Other equivalent formulations of these results can be seen in \cite[Ch. 1, Sec. 7]{Pell03} or \cite[Sec. 3.2]{ShoTam63}.

 The $(n,k)$-moment associated with the inner product \eqref{Sob-InnerP}   is defined as $m_{n,k}=\langle x^n,x^k\rangle_{\vec{\mathbf{\mu}}}$ ($n,k\geq 0$), provided the integral exists.  In the sequel, the values $\langle x^n,x^k\rangle_{\vec{\mathbf{\mu}}}$ are called {\em S-moments}.  Here, instead of a  sequence of moments, we have the  infinite matrix of moments $\mathcal{M}$ with entries $m_{n,k}=\langle x^n,x^k\rangle_{\vec{\mathbf{\mu}}}$, ($n,k\geq 0$).

Now, the Sobolev moment problem (or $S\!$-moment problem) consists of solving the following question: given an infinite matrix $ \mathcal{M}=(m_{i,j})_{i,j=0}^{\infty}$ and $d+1$ subsets $\Delta_{k} \subset \RR$ ($0\leq k \leq d$), find a set of $d+1$ measures $\{\mu_0, \mu_1, \dots, \mu_d \}$, where $\mu_d \neq 0 $ and $\sopor{\mu_k} \subset \Delta_k$,  such that $m_{i,j}=\langle x^i,x^j\rangle_{\vec{\mathbf{\mu}}}$ for $i,j=0,1, \dots$  As in the standard  case, the problem is considered \emph{definite} if it has at least one  solution, and \emph{determinate} if this solution is unique.
There are three conventional  cases of $S\!$-moment problems: the \emph{Hamburger $S\!$-moment problem} when  $\Delta_0=\cdots=\Delta_d=\RR$; the \emph{Stieltjes $S\!$-moment problem} if $\Delta_0=\cdots=\Delta_d=[0,\infty)$,  and the \emph{Hausdorff $S\!$-moment problem} for $\Delta_0=\cdots=\Delta_d=[0,1]$. Nonetheless, other combinations of the sets $\Delta_{k} \subset \RR$ ($0\leq k \leq d$) are possible too. An equivalent formulation of the Sobolev moment problem  is made for the special case $d=\infty$.

The following result was proved in \cite[Th. 1]{BaLoPi99}. In virtue of   Theorem \ref{ThCharact-HankelS}, we can now  reformulate it in the following equivalent form.

\begin{theorem}\label{Th-hp1}   Given an infinite symmetric matrix $  \mathcal{M}=(m_{i,j})_{i,j=0}^{\infty}$ and  $d+1$ subset $\Delta_{k} \subset \RR$ ($0\leq k \leq d \in \ZZp$), the $S\!$-moment problem is definite (or determinate) if and only if \; $\Phi( \mathcal{M},2d+1)=\mathcal{O}$, $\Phi( \mathcal{M},2d)\neq\mathcal{O}$ and  for each $k=0,1, \dots, d$  the Hankel matrix   $ \mathcal{H}_{k}$ (defined in \eqref{ThCharact-02}) is    such that  the classical moment problem $( \mathcal{H}_{k};\Delta_k)$ is  definite (or determinate).
\end{theorem}

Although  \cite{BaLoPi99} is devoted to the study of  the case in which $d$ is finite and the measures involved are supported on subset of the real line, there are no difficulties in extending  these results when  $ d =\infty $ or if the measures  are supported on the unit circle, as confirmed by the authors of \cite{MaSz00,MaSz01}.  The  $S\!$-moments problem for discrete Sobolev-type inner products was studied in \cite{Zag05}.

\section{Hessenberg-Sobolev matrices}\label{Sect-Hessenberg}

From the definition of the matrix of formal moments $\mathcal{M}$ in \eqref{GenMomentMatrix}, we have two immediate  consequences.

\begin{proposition}\label{PropoSimDef+}   Let $\mathcal{G}$ be a  non-degenerate Hessenberg matrix and  $\mathcal{M}_{\mathcal{G}}$ be its associated matrix of  formal moments. Then $\mathcal{M}_{\mathcal{G}}$ is  a symmetric and positive definite infinite matrix.
 \end{proposition}
  \begin{proof} Obviously,  $\dsty
\mathcal{M}_{\mathcal{G}}^{\tran}  =\left(\mathcal{T}^{-1}\; \mathcal{T}^{-\tran}\right)^{\tran}=\mathcal{T}^{-1}\; \mathcal{T}^{-\tran}=\mathcal{M}_{\mathcal{G}}.$ Moreover, for all  $n \geq 1$, $$ \det\left([\mathcal{M}_{\mathcal{G}}]_k\right) =\det\left([\mathcal{T}]_k^{-1}\right) \det\left([\mathcal{T}]_k^{-\tran}\right)=\tau_{0, 0}^2  \cdot \tau_{1, 1}^2 \cdot \ldots \cdot \tau_{k-1, k-1}^2 > 0,$$ where $\tau_{i, i}$ is the $(i,i)$-entry of $\mathcal{T}^{-1}$.
\end{proof}

The following theorem clarifies the relation  between the sequence of polynomials generated by $\mathcal{G}$ and the matrix of formal moments.

\begin{theorem}\label{Th-FormalOP}  Let $\mathcal{G}$ be a  non-degenerate Hessenberg matrix and  $\mathcal{M}_{\mathcal{G}}= \left( m_{i,j}\right)$ be its associated matrix of  formal moments.
 \begin{enumerate}[label=\arabic*)]

 \item   If  $p(x)=\sum_{i=0} ^{n_1} a_i x^i$ and  $q(x)=\sum_{j=0} ^{n_2} b_j x^j$ are two  polynomials in  $\PP$ of degree $n_1$  and $n_2$ respectively. Then, the bilinear form \eqref{MatInnerProd}  defines an inner product on $\PP$ and $\|\cdot \|_{\mathcal{G}}=\sqrt{\langle \cdot ,\cdot  \rangle_{\mathcal{G}}}$  is  the norm induced by \eqref{MatInnerProd} on $\PP$.

\item Let $m_{i,j}$ be the $(i,j)$th entry of $\mathcal{M}_{\mathcal{G}}$, as in \eqref{GenMomentMatrix}, then $m_{i,j}=\langle x^{i},x^{j} \rangle_{\mathcal{G}}$.
   \item $\{Q_n\}$, the sequence of polynomials generated by  $\mathcal{G}$,  is the  sequence of orthonormal  polynomials with respect to the  inner product \eqref{MatInnerProd}.
   \item $g_{j,k}=\langle x Q_{j}, Q_k\rangle_{\mathcal{G}}$,  where $g_{j,k}$ is the $(j,k)$-entry  of the matrix $\mathcal{G}$ given in  \eqref{InfLowerHess}, with $j\geq 0$ and $0\leq k \leq j+1$.
 \end{enumerate}
\end{theorem}

 \begin{proof}   From  Proposition \ref{PropoSimDef+},    the statement 1)  is straightforward. The assertion 2)  follows from \eqref{MatInnerProd}.

Let $\mathcal{E}_j$ be the infinite column-vector whose $i$-entry is  $\delta_{i,j}$, where $i,j\in \ZZp$.  Denote by  $Q_n(x)=\sum _{i=0} ^{n} t_{n,i} \,x^i$  the $n$th polynomial generated by $\mathcal{G}$, as in \eqref{InfiniteMatForm02}. Then, for  $j=0, \ldots , n$
\begin{align*}
\langle  Q_n, x^j \rangle_{\mathcal{G}} =&  ( t_{n,0},\cdots, t_{n,n},0,\cdots) \; \mathcal{M}_{\mathcal{G}} \; \mathcal{E}_j =\mathcal{E}_n^{\tran}\mathcal{ T}\mathcal{M}_{\mathcal{G}}  \mathcal{E}_j=\mathcal{E}_n^{\tran}\mathcal{ T} \, \mathcal{T}^{-1}\, \mathcal{T}^{-\tran}\,  \mathcal{E}_j\\
   =&  \mathcal{E}_n^{\tran}\, \mathcal{T}^{-\tran} \, \mathcal{E}_j= \tau_{n,n} \, \mathcal{E}_n^{\tran}\,  \mathcal{E}_j= \tau_{n,n} \, \delta_{n,j};
\end{align*}
where $\tau_{n,n}\neq 0$.   Furthermore,
 \begin{align*}
\langle  Q_n,  Q_n \rangle_{\mathcal{G}} =&  ( t_{n,0},\cdots, t_{n,n},0,\cdots) \; \mathcal{M}_{\mathcal{G}} \;  ( t_{n,0},\cdots, t_{n,n},0,\cdots)^{\tran}\\
  =& \mathcal{E}_n^{\tran}\, \mathcal{ T}\,\mathcal{M}_{\mathcal{G}}  \,\mathcal{T}^{\tran}  \, \mathcal{E}_n = \mathcal{E}_n^{\tran}\,\mathcal{ T}\,\mathcal{T}^{-1}\, \mathcal{T}^{-\tran}\,  \mathcal{T}^{\tran}  \, \mathcal{E}_n =1.
\end{align*}
Hence,  $Q_n$  is the $n$th  orthonormal  polynomial with respect to \eqref{MatInnerProd}. The fourth assertion is straightforward from \eqref{SobolevRR01} and the orthogonality.
\end{proof}

\begin{remark}From \eqref{SobolevRR01}-\eqref{InfiniteMatForm02}, we have that the leading coefficient of $Q_n$ is $$t_{n,n}=t_{0,0}\left(\prod_{k=0}^{n-1} g_{k,k+1}\right)^{-1}\neq 0; \quad \text{ for all \;} n \geq 1.$$  Therefore, the corresponding $n$th-monic polynomial orthogonal with respect to $\langle  \cdot,  \cdot \rangle_{\mathcal{G}}$ is $q_n=\tau_{n,n} \,Q_n$ and $\|q_n\|_{\mathcal{G}}=\tau_{n,n}=t_{n,n}^{-1}$ as in \eqref{InfiniteMatForm03}.\end{remark}

\begin{theorem}The matrices $\mathcal{G}$ and $\mathcal{M}_{\mathcal{G}}$ are  closely related  by the expression
 \begin{equation}\label{RelatHessHank}
\mathcal{G}=\mathcal{T}  \mathcal{U}\mathcal{M}_{\mathcal{G}} \mathcal{T}^{\tran}.
\end{equation}
\end{theorem}

\begin{proof} From \eqref{InfiniteMatForm02} $\dsty Q_{n}(x)=\sum_{i=0}^n t_{n,i} x^{i}$, therefore
    \begin{align*}
   g_{k,\ell}=&\langle x Q_{k}, Q_\ell\rangle_{\mathcal{G}}= \sum_{i=0}^k\sum_{j=0}^\ell  t_{k,i} \, t_{\ell,j} \,\langle x^{i+1}, x^{j}\rangle_{\mathcal{G}}=\sum_{i=0}^k\sum_{j=0}^\ell  t_{k,i} \, t_{\ell,j} \, m_{i+1,j}  \end{align*}
  which  is the $(k,\ell)$ entry of matrix $\mathcal{T}  \mathcal{U}\mathcal{M}_{\mathcal{G}}\mathcal{T}^{\tran}$ and  \eqref{RelatHessHank} is proved.
\end{proof}

\begin{theorem} \label{RelaOperadores} Let  $\mathcal{G} \in \MM$  be a non-degenerate Hessenberg matrix,  $\mathcal{M}_{\mathcal{G}} \in \MM$ be its matrix of  formal moments associated and $\eta \in \ZZp$ fixed.  Then\;
\begin{equation*}\label{EqualOperator}
\Psi( \mathcal{G},\eta)= (-1)^{n}\,\mathcal{T} \, \Phi( \mathcal{M}_{\mathcal{G}},\eta) \, \mathcal{T}^{\tran}.
\end{equation*}
\end{theorem}

\begin{proof} From \eqref{GenMomentMatrix} and \eqref{RelatHessHank}, we get that $\mathcal{G}=\mathcal{T}\mathcal{U} \mathcal{T}^{-1}$. Therefore, for each $k \in \ZZp$ we obtain
  \begin{equation}\label{RelatHessHank-2}
 \mathcal{G}^k=\mathcal{T} \mathcal{U}^{k} \mathcal{T}^{-1} \quad \text{ and }  \quad \mathcal{G}^{k\tran}=\mathcal{T}^{-\tran}  \mathcal{U}^{-k} \mathcal{T}^{\tran} .
\end{equation}

Now, from  \eqref{Phi-operator}, \eqref{Psi-operator},  \eqref{GenMomentMatrix}  and \eqref{RelatHessHank-2} it follows that

\begin{align*}
  \Psi( \mathcal{G},\eta) =& \sum_{k=0}^{\eta}(-1)^{k}\, \binom{\eta}{k} \;\mathcal{G}^{k}\; \mathcal{G}^{(\eta-k)\; \tran}=\sum_{k=0}^{\eta}(-1)^{k}\, \binom{\eta}{k} \;\mathcal{T} \,\mathcal{U}^{k} \, \mathcal{M}_{\mathcal{G}}\, \mathcal{U}^{k-\eta}\, \mathcal{T}^{\tran}\\
  =& \mathcal{T} \, \left((-1)^{n}\, \sum_{\ell=0}^{\eta}(-1)^{\ell}\, \binom{\eta}{\ell}  \,\mathcal{U}^{\eta-\ell} \, \mathcal{M}_{\mathcal{G}}\, \mathcal{U}^{-\ell} \right)\, \mathcal{T}^{\tran}.
\end{align*}
\end{proof}

\section{Favard type theorems}

One of the main problems  in the general theory of orthogonal polynomials is to characterize the non-degenerate Hessenberg matrices, for which there exists a non-discrete positive measure  $\mu$ supported on the real line such that the inner product \eqref{MatInnerProd}  can be represented as
\begin{equation}\label{3D-InnerProd}
   \langle p,q \rangle_{\mathcal{G}}=\langle p,q \rangle_{\mu}:= \int  p\,q\,d\mu.
\end{equation}
The aforementioned characterization is  the well-known \emph{Favard Theorem} (c.f. \cite{Dur93}  or \cite{AlMa01} for an overview of this theorem and its extensions), that we revisit    according to the view-point of this paper.

\begin{theorem}[Favard theorem] \label{Th-Favard}  Let $\mathcal{G}$ be a  non-degenerate Hessenberg matrix and    $\langle \cdot, \cdot \rangle_{\mathcal{G}}$ be the  inner product on $\PP$ defined by \eqref{MatInnerProd}.  Then,  there exists a non-discrete positive measure $\mu$ such that  $\langle p,q \rangle_{\mathcal{G}}= \langle p, q \rangle_{\mu}$ for all $p,q \in \PP$ if and only if \; $\Psi( \mathcal{G},1)=\mathcal{O}$.
\end{theorem}
\begin{proof} Assume that there exists a non-discrete positive measure $\mu$ such that  $\langle p,q \rangle_{\mathcal{G}}= \langle p, q \rangle_{\mu}$ for all $p,q \in \PP$, where $\mathcal{G}$ is a  non-degenerate Hessenberg matrix. From the orthogonality of the generated polynomials $Q_n$ (Theorem \ref{Th-FormalOP}) and the fact that the operator of multiplication by the variable is symmetric with respect to $\langle \cdot, \cdot \rangle_{\mu}$ ($\langle x p, q \rangle_{\mu}=\langle p, x q \rangle_{\mu}$), it is clear that $\mathcal{G}$ is a symmetric tridiagonal matrix, which is equivalent to  $\Psi( \mathcal{G},1)=\mathcal{O}$.

On the other hand, if $\mathcal{G}$ is a  non-degenerate Hessenberg matrix such that $\Psi( \mathcal{G},1)=\mathcal{O}$, we get that  $\mathcal{G}$ a  symmetric Hessenberg matrix or equivalently a non-degenerate tridiagonal matrix. From Theorem \ref{RelaOperadores},
$$ \mathcal{O}=\Phi( \mathcal{M}_{\mathcal{G}},1)=\mathcal{U} \mathcal{M}_{\mathcal{G}}-\mathcal{M}_{\mathcal{G}} \mathcal{U}^{-1},$$
i.e. $\mathcal{M}_{\mathcal{G}} $ is a Hankel matrix, which from Proposition \ref{PropoSimDef+} is positive definite. From Lemma \ref{SchoTam-lemma}, the proof is complete.
\end{proof}

Obviously, under the assumptions of  Theorems \ref{Th-FormalOP} and \ref{Th-Favard},   the sequence $\{Q_n\}$ of polynomials generated by  $\mathcal{G}$ is the sequence of orthogonal polynomials with respect to the measure $\mu$ (i.e. with respect to the inner product \eqref{3D-InnerProd}).

\begin{example} The sequence of polynomials $\{1, x, x^2 , \ldots , x^n , \ldots \}$ is generated by the non-degenerated Hessenberg matrix
$$
\mathcal{G} = \left( \begin{array}{cccc}
0 & 1 & 0 & \ldots \\
0 & 0 & 1 & \ldots \\
0 & 0 & 0 & \ldots \\
\vdots & \vdots & \vdots & \ddots
\end{array} \right),
$$
hence from  Theorem \ref{Th-FormalOP}, the sequence is orthonormal with respect to the inner product \eqref{MatInnerProd}. As $\mathcal{G}$ is a non-symmetric matrix,  $\Psi( \mathcal{G},1)\neq \mathcal{O}$.   Then, from Theorem \ref{Th-Favard} there does not exist a non-discrete positive measure  $\mu$,   such that  $\langle p,q \rangle_{\mathcal{G}}= \langle p, q \rangle_{\mu}$ for all $p,q \in \PP$.
\end{example}

\begin{proof}[{Proof of Theorem  \ref{ThFavardSobolev}.}] Let   $p(x)=\sum_{i=0} ^{n_1} a_i x^i$ and  $q(x)=\sum_{j=0} ^{n_2} b_j x^j$ be  polynomials in  $\PP$ of degree $n_1$  and $n_2$ respectively. Then, from the Sobolev inner product \eqref{Sob-InnerP} we have the representation
 \begin{equation}
\label{Sob-InnerP-2}
\langle p, q \rangle_{\vec{\mathbf{\mu}}_d}= (a_0,\cdots,a_{n_1},0,\cdots) \mathcal{S} (b_0,\cdots,b_{n_2},0,\cdots)^{\tran} = \sum_{i=0} ^{n_1} \sum _{j=0} ^{n_2} a_i \,s_{i,j}\,{b_j} \;,
\end{equation}
where $  \mathcal{S}=(s_{i,j})_{i,j=0}^{\infty}$  with $s_{i,j}=\langle x^i,x^j\rangle_{\vec{\mathbf{\mu}}_d}$, is a Hankel-Sobolev matrix of index $d$.  If $\mathcal{G}$ is a Hessenberg-Sobolev matrix of index $d \in {\ZZp}$, from   \eqref{MatInnerProd}    and \eqref{Sob-InnerP-2},   to prove that
$ \langle p, q \rangle_{\vec{\mathbf{\mu}}_d} =\langle p, q \rangle_{\mathcal{G}} $ is equivalent to prove that $\mathcal{S} =\mathcal{M}_{\mathcal{G}}$, where $\mathcal{M}_{\mathcal{G}}$ is the matrix of formal moments associated with $\mathcal{G}$.

Let $\mathcal{G}$ be a non-degenerate  Hessenberg  matrix. Assume that  there exists  $d \in {\ZZp}$ and $\vec{\mathbf{\mu}}_d= (\mu_0,\mu_1,\dots,\mu_d) \in \mathfrak{M}_{d}(\RR)$ (continuous case), such that    $\mathcal{S} =\mathcal{M}_{\mathcal{G}}$. From Theorem  \ref{Th-hp1},  $\mathcal{S}$ is a Hankel-Sobolev matrix. Therefore, combining Theorem  \ref{ThCharact-HankelS} and Theorem \ref{RelaOperadores}, we get that $\Psi( \mathcal{G},2d+1)=\mathcal{O}$ \; and \; $\Psi( \mathcal{G},2d)\neq\mathcal{O}$. Furthermore, each matrix $\mathcal{H}_k$ defined by  \eqref{ThCharact-02},   is the moment matrix of the measure $\mu_k$, which is a non-negative finite Borel measure whose  support is an infinite subset. Hence, from Lemma \ref{SchoTam-lemma} we have that $\mathcal{H}_k$ s a positive definite matrix of infinite order.

Reciprocally, let $\mathcal{G}$ be a non-degenerate  Hessenberg  matrix satisfying 1 and 2. From Theorems \ref{ThCharact-HessenS} and \ref{RelaOperadores} we conclude that $\mathcal{M}_{\mathcal{G}}$, the matrix of formal moments associated with $\mathcal{G}$, is a Hankel-Sobolev matrix of index $d$, i.e., there exist Hankel matrices $\mathcal{H}_{k}$, $k=0,\,1,\dots,\,d$, such that $$ \mathcal{M}_{\mathcal{G}}=\sum_{k=0}^{d} \left(\mathcal{U}^{-k}\;\mathcal{D}_k \; \mathcal{H}_{k} \; \mathcal{D}_k\;\mathcal{U}^{k}\right).
$$
From  Theorem \ref{Th-hp1} and Lemma \ref{SchoTam-lemma}, the $S\!$-moment  problem for $\mathcal{M}_{\mathcal{G}}$ is defined. Let $\mu_{d-k}$ be a solution of the problem of moments with respect to $\mathcal{H}_{d-k}$ for each $k=0,\,1,\dots,\,d$. If  $\langle p, q \rangle_{\vec{\mathbf{\mu}}_d}$ is as in \eqref{Sob-InnerP},  from Proposition \ref{Uniqueness-HankelS} we get $ \mathcal{S} =\mathcal{M}_{\mathcal{G}}$.
 \end{proof}

The following result may be proved in much the same way as Theorem  \ref{ThFavardSobolev}, using the appropriate assertions of Lemma \ref{SchoTam-lemma} for the case of measures supported on finite subsets.

\begin{theorem} [Favard  type theorem for discrete case] \label{ThFavardSobolevDiscrete} Let $\mathcal{G}$ be a non-degenerate  Hessenberg  matrix. Then, there exists  $d \in \ZZp$ and $\vec{\mathbf{\mu}}_d \in \mathfrak{M}_{d}(\RR)$  such that $ \langle p, q \rangle_{\vec{\mathbf{\mu}}_d} =\langle p, q \rangle_{\mathcal{G}} $ if and only if \;
\begin{enumerate}
  \item $\mathcal{G}$ is a Hessenberg-Sobolev matrix of index $d \in  \ZZp$.
  \item  The Hankel matrix    $\mathcal{H}_{0}$  defined by  \eqref{ThCharact-02}, is a positive definite matrix of infinite order and for  each  $k=0,\,1,\dots,\;d-1$;  the  matrix    $\mathcal{H}_{d-k}$   is a positive definite matrix of  order $m_k \in \ZZp$.
\end{enumerate}
\end{theorem}

The previous theorem is a refinement of \cite[Lemma 3]{Dur93}.

\begin{theorem} [Favard  type theorem for discrete-continuous case] \label{ThFavardSobolevDiscreteCont} Let $\mathcal{G}$ be a non-degenerate  Hessenberg  matrix. Then, there exists  $d \in \ZZp$ and $\vec{\mathbf{\mu}}_d \in \mathfrak{M}_{d}(\RR)$  such that $ \langle p, q \rangle_{\vec{\mathbf{\mu}}_d} =\langle p, q \rangle_{\mathcal{G}} $ if and only if \;
\begin{enumerate}
  \item $\mathcal{G}$ is a Hessenberg-Sobolev matrix of index $d \in  \ZZp$.
  \item The Hankel matrix    $\mathcal{H}_{d}$  defined by  \eqref{ThCharact-02}, is a positive definite matrix of infinite order and for  each  $k= 1,\,2\dots,\;d$;  the  matrix    $\mathcal{H}_{d-k}$   is a positive definite matrix of  order $m_k \in \ZZp$.
\end{enumerate}
\end{theorem}

 
\end{document}